\documentclass[11pt]{article}

\usepackage{amssymb, amsmath, amsthm, graphicx}
\usepackage[unicode]{hyperref}
\usepackage[left=1in,top=1in,right=1in]{geometry}
\usepackage[textsize=small]{todonotes}
\usepackage{comment}
\usepackage[capitalize]{cleveref}
\usepackage{color}
\allowdisplaybreaks

\date{}

\theoremstyle{plain}
      \newtheorem{theorem}{Theorem}[section]
      \newtheorem{lemma}[theorem]{Lemma}
            \newtheorem{claim}[theorem]{Claim}

      \newtheorem{corollary}[theorem]{Corollary}
      
      \newtheorem{conjecture}[theorem]{Conjecture}
      \newtheorem*{thm:main1}{Theorem \ref{thm:main1}}
\theoremstyle{definition}

\theoremstyle{remark}

	\newcommand{\PP}{{\mathbb P}}
	\newcommand{\EE}{{\mathbb E}}
	\newcommand{\II}{{\mathbb I}}

\newcommand{\Var}{\operatorname{Var}}

\title{On random irregular subgraphs}
\author{Jacob Fox\thanks{Department of Mathematics, Stanford University, Stanford, CA 94305. Email: {\tt jacobfox@stanford.edu}. Research supported by a Packard Fellowship and by NSF Awards DMS-1800053 and DMS-2154169.} \and Sammy Luo\thanks{Department of Mathematics, Stanford University, Stanford, CA 94305. Email: {\tt sammyluo@stanford.edu}. Research supported by NSF GRFP Grant DGE-1656518.} \and Huy Tuan Pham\thanks{Department of Mathematics, Stanford University, Stanford, CA 94305. Email: {\tt huypham@stanford.edu}. Research supported by a Two Sigma Fellowship.}}

\date{}

\begin{document}

\maketitle

\begin{abstract}
Let $G$ be a $d$-regular graph on $n$ vertices. Frieze, Gould, Karo\'nski and Pfender began the study of the following random spanning subgraph model $H=H(G)$. Assign independently to each vertex $v$ of $G$ 
a uniform random number $x(v) \in [0,1]$, and an edge $(u,v)$ of $G$ is an edge of $H$ if and only if $x(u)+x(v) \geq 1$. Addressing a problem of Alon and Wei, 
we prove that if $d = o(n/(\log n)^{12})$, then with high probability, for each nonnegative integer $k \leq d$, there are $(1+o(1))n/(d+1)$ vertices of degree $k$ in $H$.  
\end{abstract}

\section{Introduction}

Given a graph $H$ and a nonnegative integer $k$, let $m(H,k)$ be the number of vertices in $H$ with degree $k$, and let $m(H)=\max_k m(H,k)$. It is well-known that in any finite simple graph $H$, there is some pair of vertices with the same degree, that is, $m(H)\geq 2$. Alon and Wei \cite{AW} initiated the study of the following question: Given a $d$-regular graph $G$ on $n$ vertices, can we find a spanning subgraph $H$ of $G$ such that $m(H)$ is as small as possible? Since such a subgraph $H$ has maximum degree at most $d$, we clearly have $m(H)\geq \frac{n}{d+1}$. Alon and Wei showed that this lower bound is in fact nearly optimal by proving that every $d$-regular graph $G$ has a spanning subgraph $H$ satisfying $m(H)\leq (1+o_n(1))\frac{n}{d+1}+2$ (in fact, it suffices to assume that $G$ has minimum degree $d$). They conjectured that for regular graphs, one can guarantee a stronger conclusion, that $m(H,k)$ is close to $\frac{n}{d+1}$ for all $0\le k\le d$. 
\begin{conjecture}[{\cite{AW}, Conjecture~1.1}]\label{conj:AW1}
Every $d$-regular graph $G$ on $n$ vertices contains a spanning subgraph $H$ such that 
\begin{equation*}
    \big | m(H,k)-\frac{n}{d+1}\big | \leq 2 \qquad \text{for all } 0\leq k\leq d.
\end{equation*}
\end{conjecture}
Loosely speaking, we describe a subgraph $H$ as \emph{irregular} when each $m(H,k)$ is close to $\frac{n}{d+1}$, as the degrees of its vertices are close to as uniformly distributed as possible. While the argument of Alon and Wei provides a nearly optimal upper bound to $m(H)$, it does not yield even the following weaker asymptotic version of Conjecture \ref{conj:AW1} that they ask.
\begin{conjecture}[\cite{AW}]\label{conj:AW2}
Let $d=o(n)$. Then every $d$-regular graph $G$ on $n$ vertices contains a spanning subgraph $H$ such that 
\begin{equation}\label{eqn:asymp}\tag{$*$}
    m(H,k)=(1+o(1))\frac{n}{d+1} \qquad \text{for all } 0\leq k\leq d.
\end{equation}
\end{conjecture}
While both conjectures remain open, significant progress has been made on Conjecture \ref{conj:AW2} in certain ranges of $d$.
In fact, an earlier work of Frieze, Gould, Karo\'nski and Pfender \cite{Frieze} already gave a method for constructing a random spanning subgraph $H$ satisfying \eqref{eqn:asymp} whenever $d\leq (n/\log n)^{1/4}$. This construction was motivated by their study of the \emph{irregularity strength} $s(G)$ of a graph $G$, defined as the smallest integer $w$ such that one can assign integer weights in $[w]$ to every edge of $G$ so that the sums of the weights incident to each vertex are distinct. Faudree and Lehel conjectured in \cite{FL} that there is some absolute constant $C$ such that $s(G)\leq \frac{n}{d}+C$ whenever $G$ is a $d$-regular graph on $n$ vertices; see \cite{JP, PW1, PW2} for recent progress on this conjecture.

The natural construction considered in \cite{Frieze} is as follows: For each vertex $v \in V(G)$, pick a weight $x(v) \in [0,1]$ uniformly and independently at random. Then, for each edge $(u,v) \in E(G)$, we keep it as an edge of $H$ if and only if $x(u)+x(v) \geq 1$. We will henceforth refer to this construction as the \emph{irregular random subgraph} model. It is easy to check that this model satisfies at least the basic properties one would expect from a construction of a very irregular subgraph. Indeed, fix $k \in \{0,1,\ldots,d\}$ and consider the random variable $X=m(H,k)$. For a vertex $v$, let $\II_v$ be the indicator random variable for the event that $v$ has degree $k$ in $H$. Then, 
$X=\sum_{v \in V(G)}\II_v$. Observe that $v$ has degree $k$ in $H$ if and only if, among the set of $d+1$ numbers consisting of  $1-x(v)$ and the $d$ numbers $x(u)$ with $u$ a neighbor of $v$ in $G$, $1-x(v)$ is the $(k+1)^{th}$ largest. It follows that $v$ has degree $k$ in $H$ with probability $1/(d+1)$. By linearity of expectation, $\mathbb{E}[X]=n/(d+1)$. To give an asymptotically tight bound on $m(H,k)$, it then remains to show that the random variable $X$ is well-concentrated around this mean. The authors of \cite{Frieze} do so via an application of the Azuma-Hoeffding inequality. 
In \cite{AW}, Alon and Wei show a larger range of values of $d$ over which \eqref{eqn:asymp} holds with high probability in this construction by using a more careful concentration argument. They start by equitably partitioning the vertices of $G$ into $O(d^2)$ disjoint sets $V_i$ such that within each $V_i$, each pair of vertices has distance at least $3$, and thus the degrees in $H$ of all vertices in $V_i$ are independent. This allows them to apply Chernoff's inequality to show that \eqref{eqn:asymp} holds with high probability whenever $d=o((n/\log n)^{1/3})$. 

In this paper, we take a more direct approach to studying the concentration properties of this irregular random subgraph model, showing the desired asymptotic irregularity property for an even wider range of values of $d$. First, we take a natural approach of applying the second moment method to study the variance of the random variable $X$. We show that $\Var(X)\leq \frac{17n}{d+1}=17\EE[X]$, implying by Chebyshev's inequality and the union bound that \eqref{eqn:asymp} holds with high probability whenever $d=o(n^{1/2})$. This argument is detailed in Section~\ref{sec:2ndmoment}. However, one might conjecture much stronger concentration bounds to hold due to exponential-type decay in the distribution of $X$. By using a much more technical argument employing the martingale Bernstein inequality, we manage to extend the result to all $d=o(n/(\log n)^{12})$, showing the following main theorem.
\begin{theorem}\label{thm:main}
Let $G$ be a $d$-regular graph on $n$ vertices with $d=o(n/(\log n)^{12})$. Then with high probability, the irregular random subgraph $H=H(G)$ satisfies that for $0\leq k \leq d$, the number of vertices with degree $k$ in $H$ is $(1+o(1))n/(d+1)$.
\end{theorem}

In particular, this theorem confirms Conjecture \ref{conj:AW2} for all $d = o(n/(\log n)^{12})$. We conjecture that the irregular random subgraph $H$ satisfies Property \ref{eqn:asymp} with high probability for all $d = o(n/\log n)$. On the other hand, it is likely that for $d=\omega(n/\log n)$, the irregular random subgraph model does not satisfy Property \ref{eqn:asymp} with high probability, and new ideas would be required to show Conjecture \ref{conj:AW2} in the full range.   

The proof of Theorem~\ref{thm:main} is in Section~\ref{sec:bernstein}. Since these two arguments are more or less independent, they can be read in any desired order. 

\section{The second moment argument}
\label{sec:2ndmoment}

In this section, we prove Theorem~\ref{thm:main} in the case $d=o(n^{1/2})$ through a second moment argument. We establish the following bound on the variance of the random variable $X=m(H,k)$.

\begin{theorem}
\label{thm:varX}
Let $G$ be a $d$-regular graph on $n$ vertices, and let $H=H(G)$ be the irregular random subgraph. Then for all $0\leq k\leq d$, the variance of the random variable $X=m(H,k)$ is at most $17\frac{n}{d+1}$.
\end{theorem}
Recall that $X=\sum_{v\in V(G)}\II_v$, so
\begin{equation}
\label{eqn:2ndmoment}
\Var(X)=\mathbb{E}[X^2]-\mathbb{E}[X]^2=n/(d+1)+\sum_{u \not = v} \mathbb{E}[\II_u \II_v]-(n/(d+1))^2.    
\end{equation}
We begin by establishing the following key lemma. It gives an upper bound on the probability that two given vertices each have degree $k$ in the irregular random subgraph model that depends on the codegree of the two vertices in $G$. Recall that the codegree $d(u,v)$ is the number of common neighbors of vertices $u$ and $v$ in graph $G$. For a vertex $u$ of $G$, let $N(u)$ denote the neighborhood of $u$ in $G$ and $N_+(u) = N(u)\cup \{u\}$. We also use $u\sim v$ to denote that vertices $u$ and $v$ are adjacent in graph $G$. 
\begin{lemma}
\label{lem:2ndmoment}
For any two vertices $u$ and $v$,
\[
\mathbb{E}[\II_u\II_v] \le \frac{1}{(d+1)^2}\left(1+\frac{16d(u,v)}{d+1}\right).
\]
\end{lemma}
\begin{proof}
We have $\II_u=\II_v=1$ if and only if, for each $w\in \{u,v\}$, $1-x(w)$ is the $(k+1)^{th}$ largest number among $\{1-x(w), x(z) : z\sim w\}$. Let $N_u = N(u) \setminus (N(v) \cup \{v\})$, $N_v = N(v)\setminus (N(u)\cup \{u\})$, and $N_{u,v} = N(u)\cap N(v)$. Let $t = d(u,v)$. We consider two cases based on the value of $t$.

{\noindent \bf Case 1:} $t \ge d/2$.

 Condition on an arbitrary realization of $x(z)$ for all $z\in N_{u}\cup N_{v}$, and a realization of the set of values $S = \{1-x(u), 1-x(v), x(z) : z\in N_{u,v}\}$ (but not the specific value of each $x(z)$ in this set, nor of $x(u)$ or $x(v)$). As $1-x(u), 1-x(v), x(z)$ for $z\in N(u)\cup N(v)\setminus \{u,v\}$ are independent and identically distributed, the conditional distribution of $(1-x(u), 1-x(v), x(z) : z\in N(u) \cap N(v))$ is given by a uniformly random permutation of the values in $S$. 

Let $r_u$ and $r_v$ be the positions of $1-x(u),1-x(v)$ realized in a random permutation of the values in $S$ (when $S$ is ordered in non-increasing order). When $u\nsim v$, there are at most two choices of $r_u$ and $r_v$ so that $\II_u=\II_v=1$, depending on whether $1-x(u)$ and $1-x(v)$ need to be in the same position or different positions relative to the other $x(z)$. When $u\sim v$, depending on whether $x(u)+x(v)\geq 1$, the positions of $1-x(u)$ and $1-x(v)$ among $\{1-x(u)\}\cup N_u \cup N_{u,v}$ and $\{1-x(v)\}\cup N_v \cup N_{u,v}$ respectively must either both be $k^{th}$ largest or both be $(k+1)^{th}$ largest. For a fixed set of values $S$, however, only one of these two cases can yield the correct final ordering of $\{1-x(u), x(z) : z\sim u\}$ and $\{1-x(v), x(z) : z\sim v\}$, since going from $k^{th}$ largest to $(k+1)^{th}$ largest means choosing strictly smaller values for $1-x(u)$ and $1-x(v)$ among $S$, which increases $x(u)+x(v)$. Thus, in any case, there are at most two choices of $r_u$ and $r_v$ so that $\II_u=\II_v=1$. Thus, the conditional probability that $\II_u=\II_v=1$ is at most $2t!/(t+2)! < 2/t^2$.

{\noindent \bf Case 2:} $t < d/2$.

Condition on an arbitrary realization of $x(z)$ for all $z\in N_{u,v}$, and for each $w\in \{u,v\}$, a realization of the set of values $S_w = \{1-x(w), x(z) : z\in N_{w}\}$ (but not the specific value of each $x(z)$ in this set, nor of $x(w)$). As $1-x(u), 1-x(v), x(z)$ for $z\in N(u)\cup N(v)\setminus \{u,v\}$ are independent and identically distributed, the conditional distributions of $(1-x(u), x(z) : z\in N_u)$ and $(1-x(v), x(z) : z\in N_v)$ are uniformly and independently distributed among permutations of $S_u$ and $S_v$. 

Let $r_u$ and $r_v$ be the positions of $1-x(u),1-x(v)$ realized in random permutations of the values in $S_u$ and $S_v$ respectively. If $u$ and $v$ are not adjacent, there is a unique choice of $(r_u,r_v)$ so that $\II_u=\II_v=1$, while if $u$ and $v$ are adjacent, as in the previous case we have two possible pairs of positions, but at most one pair can be valid given the set of values in $S_u$ and $S_v$. Thus, the conditional probability that $\II_u=\II_v=1$ is at most $1/(d-t)^2$.

Hence, in both cases, we obtain
\[
\mathbb{E}[\II_u\II_v] \le \frac{1}{(d+1)^2}\left(1+\frac{16t}{d+1}\right).
\]
\end{proof}

From Lemma~\ref{lem:2ndmoment}, the desired bound on $\Var(X)$ follows.

\begin{proof}[Proof of Theorem~\ref{thm:varX}]
By Lemma~\ref{lem:2ndmoment},
\begin{align*}
\mathbb{E}[X^2] &\le \frac{n}{d+1} + \sum_{u\ne v}\frac{1}{(d+1)^2}\left(1+\frac{16d(u,v)}{d+1}\right)\\
&= \frac{n}{d+1}+\frac{16}{(d+1)^3}\sum_{u\ne v} d(u,v) + \frac{n(n-1)}{(d+1)^2}\\
&= \frac{n}{d+1}+\frac{16}{(d+1)^3}\cdot nd(d-1) + \frac{n(n-1)}{(d+1)^2}\\
&\le \frac{17n}{d+1} + \left(\frac{n}{d+1}\right)^2,
\end{align*}
where in the second equality, we have used the identity $\sum_{u\ne v}d(u,v) = 2\sum_{w} \binom{|N(w)|}{2} = nd(d-1)$. Hence, by \eqref{eqn:2ndmoment},
\begin{equation}\label{eq:var-X}
\textrm{Var}(X) =\mathbb{E}[X^2]-\mathbb{E}[X]^2\le \frac{17n}{d+1}.\qedhere
\end{equation}
\end{proof}
This immediately implies the following desired result.
\begin{corollary}\label{cor:sqrtn}
Theorem~\ref{thm:main} holds in the case $d=o(n^{1/2})$.
\end{corollary}
\begin{proof}
By Chebyshev's inequality and the bound (\ref{eq:var-X}), the probability that $|X-\mathbb{E}[X]| \geq z$ with $z=\epsilon n/(d+1)$ is at most 
\[
\frac{\textrm{Var}(X)}{z^2}\le \frac{17(d+1)}{\epsilon^2 n}.
\]
By the union bound over the $d+1$ possible values of $k$, the probability that there exists $k$ such that $|m(H,k)-n/(d+1)| > \epsilon n/(d+1)$ is at most $\frac{17(d+1)^2}{\epsilon^2n}$. For $d \le \epsilon^2 \sqrt{n}/5 - 1$, we then have that with probability at least $1-\epsilon^2$, $|m(H,k)-n/(d+1)| < \epsilon n/(d+1)$ for all $k \le d$. By taking $\epsilon \to 0$ sufficiently slowly, we obtain that for $d=o(n^{1/2})$, with high probability,  $m(H,k)=(1+o(1))n/(d+1)$ for all $k\le d$, as desired.
\end{proof}

In the next section we will see that the range of $d$ can be significantly improved by replacing the use of Chebyshev's inequality with a much stronger concentration result that holds for the random variable $X$. 

\section{Exponential concentration via martingale Bernstein inequality}
\label{sec:bernstein}

In this section, we use the martingale Bernstein inequality in order to show that the random variable $X=m(H,k)$ defined in the introduction, which is the number of vertices of degree $k$ in the random subgraph $H=H(G)$, is exponentially concentrated on its expected value. This will allow us to prove Theorem~\ref{thm:main} in full generality. Since Corollary \ref{cor:sqrtn} handles the case $d=o(n^{1/2})$, to cover the remaining range of $d$, in this section, we can and will assume that $d = \omega(\log n)$.

We aim to get exponential concentration for $X$ based on the following Bernstein concentration bound for sums of martingale differences. This is essentially due to Freedman \cite{Freedman}, but we will use the following version, which follows from combining Corollary~2.3 and Remark~2.1 in the paper by Fan, Grama and Liu \cite{FGL}.

\begin{lemma}\label{lem:bernstein}
Let $\mathcal{F}_1\subseteq \mathcal{F}_2\subseteq \cdots \subseteq \mathcal{F}_n$ be an increasing sequence of $\sigma$-algebras. Given random variables $Y_1,\dots,Y_n$ such that $\mathbb{E}[Y_j | \mathcal{F}_{j-1}] = 0$, let $X_i = \sum_{j=1}^{i}Y_j$, and let $M_i = \sum_{j=1}^{i} \mathbb{E}[Y_j^2 \mid \mathcal{F}_{j-1}]$. %
Then for any $z\geq 0$, $a>0$ and $L>0$,
\[
\mathbb{P}[\max_{i\le n}|X_i - X_0| \ge z, M_n \le L] \le \exp\left(-\frac{1}{2} \frac{(z/a)^2}{L/a^2+z/a}\right) + \mathbb{P}[\max_{i\le n} |Y_i| > a].
\]
\end{lemma}

Let $\pi$ be a permutation of the vertices of the graph chosen uniformly at random. In the following, we identify vertices of the graph with integers in $[n]$ according to the permutation $\pi$. Let ${\cal F}_i = \sigma(x(u):u\le i)$ denote the $\sigma$-algebra generated by the random variables $x(u)$ with $u\le i$. Let $X_i = \mathbb{E}[X \mid \mathcal{F}_i]$, so $\mathbb{E}[X] = X_0$ and $X = X_{n}$. Let $Y_i = X_i - X_{i-1}$ and let $M_i = \sum_{j=1}^{i} \mathbb{E}[Y_j^2 \mid \mathcal{F}_{j-1}]$. In order to apply Lemma \ref{lem:bernstein} in our setting, we need the following two lemmas. 

\begin{lemma}\label{lem:sup-Yi}
There exists a constant $C_1>0$ such that, with probability at least $1-n^{-50}$, we have $|Y_i| \leq C_1 \log n$ for all $i$.
\end{lemma}

\begin{lemma}\label{lem:var-proxy}
Assume that $d=\omega(\log n)$. There exists a constant $C_2>0$ such that, with probability at least $1-n^{-100}$, we have that $M_n \leq C_2 (\log n)^{11} n/d$. 
\end{lemma}

We now give the proof of Theorem \ref{thm:main} assuming Lemma \ref{lem:sup-Yi} and Lemma \ref{lem:var-proxy}. 

\begin{proof}[Proof of Theorem~\ref{thm:main}]
Let $C_0=\max(C_1,C_2)$. 
For $z>0$, let $\mathcal{E}_1$ be the event that $|X_n-X_0|>z$, $\mathcal{E}_2$ the event that $M_n \le C_0(\log n)^{11}n/d$, and $\mathcal{E}_3$ the event that $\sup_{j\le n}|Y_j| \le C_0(\log n)$.
By Lemma \ref{lem:bernstein} with $a=C_0(\log n)$ and $L:= C_0(\log n)^{11} n/d$, we have for some $c>0$ that 
\begin{equation}\label{eq:prob-E123}
\mathbb{P}[\mathcal{E}_1 \wedge \mathcal{E}_2] \le \exp\left(-\frac{cz^2}{(\log n)^{11}n/d+(\log n)z}\right) +\PP[\overline{\mathcal{E}_3}]. \nonumber
\end{equation}
Thus, 
\begin{align*}
\mathbb{P}[\mathcal{E}_1] &\le \mathbb{P}\left[\mathcal{E}_1\wedge \mathcal{E}_2\right] + \mathbb{P}[\overline{\mathcal{E}_2}] \\
&\le \exp\left(-\frac{cz^2}{(\log n)^{11}n/d+(\log n)z}\right) + 2n^{-50},
\end{align*}
where we used that $\mathbb{P}[\overline{\mathcal{E}_2}]\le n^{-50}$ by Lemma \ref{lem:var-proxy} and $\mathbb{P}[\overline{\mathcal{E}_3}]\le n^{-50}$ by Lemma \ref{lem:sup-Yi}. 

In particular, for any $\epsilon>0$, if $z=\epsilon n/(d+1)$ with $d<c'\epsilon^2 n/(\log n)^{12}$, then we have 
\[
\mathbb{P}[|X_n - X_0|>z] < 1/n. 
\]
Note that $X_0 = \mathbb{E}[X] = n/(d+1)$. By taking $\epsilon \to 0$ sufficiently slowly and using the union bound over the $d+1$ nonnegative integers $k \leq d$, we obtain the following conclusion. With high probability, for all $k\le d$, the number of vertices with degree exactly $k$ in $H$ is $(1+o(1))n/(d+1)$.
\end{proof}

It thus remains to prove Lemmas \ref{lem:sup-Yi} and \ref{lem:var-proxy}. This is done in the next two subsections.

\subsection{Supremum of the martingale differences}
In this subsection, we give the proof of Lemma \ref{lem:sup-Yi}, which gives an upper bound on the supremum of the martingale difference $Y_i$. 

Here and throughout, we denote by $C$ a sufficiently large absolute constant, $\kappa = C \log n$, and $k_+ = C \max(k, \kappa)$. %

We will use the following three simple claims in the proof, which are standard applications of the Chernoff bound. 

\begin{claim}\label{claim:chernoff}%
Let $h>0$. Let $I\subseteq \mathbb{R}$ be an interval of positive length. Let $x_1,\dots,x_m$ be independent and uniformly distributed random variables in $I$. For a given subinterval of $I$ of length $h|I|/m$, the probability that the number of elements $x_i$ lying in the interval is not contained in $[h\pm \sqrt{\kappa \max(h,\kappa)}]$ is at most $\exp(-\kappa/3)$. 
\end{claim}

\begin{proof}
The number of variables $x_i$ contained in a fixed interval of length $h|I|/m$ is a sum of $m$ independent $\textrm{Ber}(h/m)$, which has expectation $h$. By the multiplicative form of the Chernoff bound, the probability that this deviates from $h$ by more than $\ell := \sqrt{\kappa \max(h,\kappa)}$ is at most 
\[
\exp\left(-\min\left(\left(\ell/h\right)^2,\ell/h\right)h/3\right) \le \exp(-\kappa/3).\qedhere
\]
\end{proof}

\begin{claim}\label{claim:high-prob-2}
Let $h > 1$. Let $x_1,\dots,x_m$ be independent and uniformly distributed random variables in an interval $I$. For each subinterval of length $h|I|/m$, the probability that there are more than $2h$ variables $x_i$ in this interval is at most $\exp(-h/3)$.
\end{claim}

\begin{proof}
The number of variables $x_i$ contained in a fixed subinterval of $I$ of length $h|I|/m$ is a sum of $m$ independent $\textrm{Ber}(h/m)$, which has expectation $h$. By the multiplicative form of the Chernoff bound, the probability that this is more than $2h$ is at most $\exp(-h/3)$. 
\end{proof}

\begin{claim}\label{claim:high-prob-1}
Let $h > 1$. Let $x_1,\dots,x_m$ be independent and uniformly distributed random variables in an interval $I$. The probability that there exists an interval of length at least $h|I|/m$ with no element $x_i$ is at most $3m\exp(-h/3)$.%
\end{claim}
\begin{proof}
Clearly, we may assume $h \leq m$. Partition $I$ into $t=\lceil 2m/h \rceil$ equal length intervals. This guarantees that every interval of length $h|I|/m$ contains at least one of these $t$ intervals in the partition. The probability that 
a given one of these $t$ intervals contains none of $x_1,\dots,x_m$ is $(1-1/t)^m \leq e^{-m/t} \leq e^{-h/3}$. Taking the union bound over all $t$ intervals, the probability that at least one of the $t$ intervals contains none of $x_1,\dots,x_m$ is at most 
$te^{-h/3} \leq 3m e^{-h/3}$.
\end{proof}

Recall that for a vertex $u$ of $G$, $N(u)$ denotes the neighborhood of $u$ in $G$ and $N_+(u) = N(u)\cup \{u\}$. 

\begin{proof}[Proof of Lemma \ref{lem:sup-Yi}]
Observe that 
\begin{align*}
    |Y_u| &= |\mathbb{E}[X|\mathcal{F}_u] - \mathbb{E}[X|\mathcal{F}_{u-1}]| \\
    &= \left|\sum_{v \in N_+(u)}(\mathbb{E}[\II_v | \mathcal{F}_u]-\mathbb{E}[\II_v | \mathcal{F}_{u-1}])\right|\\
    &\le \mathbb{E}\left[\sum_{v\in N(u)} \II_v \bigg | \mathcal{F}_{u}\right] + \mathbb{E}\left[\sum_{v\in N(u)} \II_v\bigg | \mathcal{F}_{u-1}\right]+1.
\end{align*}

Let $S(u) = \sum_{v\in N(u)}\II_v$. Then it suffices to show that $S(u) \le 10\kappa = 10C\log n$ %
with probability at least $1-3n^{-100}$, as combining with $S(u) \le d \le n$, we have that $\mathbb{E}\left[S(u) | \mathcal{F}_{u}\right] \le 20 \kappa$ with probability at least $1-n^{-99}$. Indeed, if  $\mathbb{E}[S(u)|\mathcal{F}_u] > 20\kappa$, then $\mathbb{E}[\mathbb{I}(S(u)>10\kappa)|\mathcal{F}_u] > 1/n$, and the probability of such an event is at most $n^{-99}$. Thus, we must have that $\mathbb{P}[\mathbb{E}[S(u)|\mathcal{F}_u] > 10\kappa] < n^{-99}$.

First consider the case $|k/d-1/2| > \sqrt{\kappa k_+}/d$.  %
Let $\mathcal{G}$ be the $\sigma$-algebra generated by $\mathbb{I}(x(v)\notin [k/d\pm \sqrt{\kappa k_+}/d])x(v)$ for all $v\in N(u)$. Condition on $\mathcal{G}$ (i.e., on the set of $v\in N(u)$ with $|x(v)-k/d|>\sqrt{\kappa k_+}/d$ and on the value of $x(v)$ for each such $v$) %
satisfying the following properties:
\begin{enumerate}
    \item All vertices $v$ in $N(u)$ with $H$-degree $k$ %
    have $|x(v)-k/d|\le  \sqrt{\kappa k_+}/d$ with probability at least $1-n^{-50}$ conditional on $\mathcal{G}$. %
    \item The number of vertices $v\in N(u)$ such that $|x(v)-k/d|\le \sqrt{\kappa k_+}/d$ is at most $4\sqrt{\kappa k_+}$. %
    \item The values $x(v)$ for $v\in N(u)$ satisfying $x(v) \notin [k/d \pm \sqrt{\kappa k_+}/d]$ partition $[0,1]\setminus [k/d \pm \sqrt{\kappa k_+}/d]$ into consecutive intervals of length at most $\kappa/d$. 
\end{enumerate}
\begin{claim}\label{claim:properties-G}
The $\sigma$-algebra $\mathcal{G}$ satisfies the above three properties with probability at least $1-n^{-100}$. 
\end{claim}
\begin{proof}
By Claim \ref{claim:chernoff}, with probability at least $1-n^{-100}/4$, for each vertex $v$, the number of neighbors $v'$ of $v$ with $x(v') \geq 1-x(v)$ is $dx(v) \pm \sqrt{\kappa \max(\kappa,dx(v))}$. Under such an event, if $v \in N(u)$ has $H$-degree $k$, then $k\in [dx(v)\pm \sqrt{\kappa \max(\kappa,dx(v))}]$. This implies that $x(v) \le 2k/d + \kappa/d$, from which we obtain that $|x(v) - k/d| \le 2\frac{\sqrt{\kappa}}{d} \sqrt{\max(\kappa,dx(v))} \le \sqrt{\kappa k_+}/d$. %

Claim \ref{claim:high-prob-2} implies that the second property holds with probability at least $1-d\exp(-2\sqrt{\kappa k_+}/3) > 1-n^{-100}/4$. Claim \ref{claim:high-prob-1} implies that the third property holds with probability at least $1-3d\exp(-\kappa/3)>1-n^{-100}/4$. %
\end{proof}

Let $S$ be the set of vertices $v \in N(u)$ with $|x(v)-k/d|\le \sqrt{\kappa k_+}/d$. By our assumption that $\mathcal{G}$ satisfies the second property above, $|S|\leq 4 \sqrt{\kappa k_+}$. Partition $[k/d-\sqrt{\kappa k_+}/d,k/d+\sqrt{\kappa k_+}/d]$ into consecutive intervals with endpoints at each value in the set $\{1-x(w):\: w\in N(u),\: x(w) \in [1-k/d-\sqrt{\kappa k_+}/d,1-k/d+\sqrt{\kappa k_+}/d]\}$, which is $\mathcal{G}$-measurable because $|k/d-1/2|>\sqrt{\kappa k_+}/d$. Conditioned on $\mathcal{G}$, for each $v\in S$, there is at most one interval $I_v$ in this partition for which $v$ can have $H$-degree $k$ only if $x(v)\in I_v$. Furthermore, the third property above implies that $|I_v|\leq \kappa/d$. Conditional on $\mathcal{G}$, the random variables $x(v)$ for $v\in S$ are independently and uniformly distributed in the interval $[k/d-\sqrt{\kappa k_+/d},k/d+\sqrt{\kappa k_+/d}]$. Note that $|I_v|/(2\sqrt{\kappa k_+}/d) \le 2\kappa/|S|$. %
By Claim \ref{claim:high-prob-2}, the number of $v\in S$ with $x(v)\in I_v$ is at most $4\kappa$ with probability at least $1-\exp(-\kappa/3) \ge 1-n^{-100}$. Hence, with probability at least $1-2n^{-100}$, there are at most $4\kappa$ neighbors of $u$ with degree $k$. 

Next consider the case $|k/d-1/2| \le \sqrt{\kappa k_+}/d$. %
Pick three consecutive intervals $J_1,J_2,J_3$ whose union contains $[k/d\pm \sqrt{\kappa k_+}/d]$, with the properties that $J_2$ has length at most $\kappa/d$ and $\frac{1}{2}\in J_2$, and $J_1$ and $J_3$ each has length $\sqrt{\kappa k_+}/d$. It follows that $1-x \notin J_j$ for $x\in J_j$ and $j=1,3$. Let $\mathcal{G}$ be the $\sigma$-algebra generated by $\mathbb{I}(x(v)\notin J_1\cup J_2 \cup J_3)x(v)$ for all $v\in N(u)$, and let $\mathcal{G}_j$ be the $\sigma$-algebra generated by $\mathbb{I}(x(v)\notin J_j)x(v)$. Condition on $\mathcal{G}$ and $\mathcal{G}_2,\mathcal{G}_3$ satisfying the following properties:
\begin{enumerate}
    \item All vertices $v$ in $N(u)$ with $H$-degree $k$ lie in $J_1\cup J_2\cup J_3$ with probability at least $1-n^{-50}$ conditional on $\mathcal{G}$.
    \item There are at most $2\kappa$ vertices $v\in N(u)$ with $x(v)\in J_2$, and at most $2 \sqrt{\kappa k_+}$ vertices with each of $x(v)\in J_1$ or $x(v)\in J_3$. 
    \item The values $x(v)$ for $v\in N(u)$ satisfying $x(v) \notin [k/d \pm \sqrt{\kappa k_+}/d]$ or $x(v)\in J_2\cup J_3$ partition $[0,1]\setminus J_1$ into consecutive intervals of length at most $\kappa/d$.
\end{enumerate}
As in Claim \ref{claim:properties-G}, we can verify that $\mathcal{G},\mathcal{G}_2,\mathcal{G}_3$ satisfy these properties with probability at least $1-n^{-100}$. 

Let $S$ be the set of vertices $v$ with $x(v) \in J_1$ (so they are the only remaining vertices with $x(v)$ not measurable in $\mathcal{G},\mathcal{G}_2,\mathcal{G}_3$). By the second property conditioned on above, $|S|\leq 2 \sqrt{\kappa k_+}$. 
By the third property above, the vertices $v$ with $x(v)\notin J_1$ partition $[0,1]\setminus J_1$ into consecutive intervals each of length at most $\kappa/d$. As before, since $1-x(v)\notin J_1$ for any $v\in S$, for each $v\in S$, conditional on $\mathcal{G}$, $\mathcal{G}_2$, and $\mathcal{G}_3$, there is at most one interval $I_v$ in this partition for which $v$ can have $H$-degree $k$ only if $x(v) \in I_v$. As before, $|I_v|/|J_1|\le 2\kappa/|S|$, and by Claim \ref{claim:high-prob-2}, the number of $v\in S$ with $x(v)\in I_v$ conditional on $\mathcal{G},\mathcal{G}_2,\mathcal{G}_3$ is at most $4\kappa$ with probability at least $1-n^{-100}$. %
Thus, the number of vertices with $x(v)\in J_1$ and $H$-degree $k$ is at most $4\kappa$ with probability at least $1-n^{-100}$. A similar argument shows that the number of vertices with $x(v)\in J_3$ and $H$-degree $k$ is at most $4\kappa$ with probability at least $1-n^{-100}$. Finally, since there are at most $2\kappa$ vertices $v\in N(u)$ with $x(v)\in J_2$, with probability at least $1-3n^{-100}$, there are at most $10\kappa$ neighbors of $u$ with $H$-degree $k$, completing the proof. %
\end{proof}

\subsection{The variance proxy}

In this subsection, we give the proof of Lemma \ref{lem:var-proxy}. Recall that 
\begin{equation}\label{eq:Mn}
M_n = \sum_{j=1}^{n} \mathbb{E}[Y_j^2 | \mathcal{F}_{j-1}].
\end{equation}
We first develop several preliminary estimates that will be useful for the proof of Lemma \ref{lem:var-proxy}. We have that
\begin{align*}
    \mathbb{E}[Y_j^2 \mid \mathcal{F}_{j-1}] &= \mathbb{E}\left[\left(\sum_{u \in N(j)}( \mathbb{E}[\II_u \mid \mathcal{F}_{j}]-\mathbb{E}[\II_u \mid \mathcal{F}_{j-1}])\right)^2 \bigg | \mathcal{F}_{j-1}\right].
\end{align*}
For an integer $h\in [0,t]$, define 
\[
p(x,t,h) = \binom{t}{h} x^{h} (1-x)^{t-h}. 
\]
If $h < 0$ or $h>t$ then by convention we set $p(x,t,h)=0$.
Note that $p(x,t,h)$ is the probability that $u$ has degree $k$ conditioned on the events that $x(u)=x$, and among $d-t$ revealed neighbors $v$ of $u$ in $G$, there are exactly $k-h$ neighbors with $x(v) \ge 1-x$. %

For each $j$, let $t(u,j)$ be the number of neighbors of $u$ that arrive after $j$ in the ordering. 
We then have for $u$ adjacent to $j$ in $G$ that 
\[
\mathbb{E}[\II_u \mid \mathcal{F}_{j}] = \mathbb{E}[p(x(u), t(u,j), k-h(u, x(u),j)) \mid \mathcal{F}_{j}],
\]
where $h(u,x(u),j)$ is the number of neighbors $v$ of $u$ arriving by time $j$ with $x(v)\ge 1-x(u)$. In particular, if $u\le j$, then 
\[
\mathbb{E}[\II_u \mid \mathcal{F}_{j}] = p(x(u), t(u,j), k-h(u, x(u), j)).
\]

Note that for a neighbor $u$ of $j$, $u$ has degree $k$ if $x(j) \ge 1-x(u)$ and $k-1-h(u,x(u),j-1)$ neighbors $v'$ of $u$ coming after $j$ have $x(v')\ge 1-x(u)$; or $x(j)< 1-x(u)$ and $k-h(u,x(u),j-1)$ neighbors $v'$ of $u$ coming after $j$ have $x(v')\ge 1-x(u)$. Thus,
\begin{align*}
\mathbb{E}[\II_u \mid \mathcal{F}_{j-1}] &= \mathbb{E}[p(x(u), t(u,j), k-h(u, x(u),j-1)) \cdot (1-x(u)) \mid \mathcal{F}_{j-1}] \\
&\qquad \qquad + \mathbb{E}[p(x(u), t(u,j), k-1-h(u, x(u),j-1)) \cdot x(u) \mid \mathcal{F}_{j-1}],
\end{align*}
while 
\begin{align*}
\mathbb{E}[\II_u \mid \mathcal{F}_{j}] &= \mathbb{E}[p(x(u), t(u,j), k-h(u, x(u),j-1)) \cdot \mathbb{I}(x(j) < 1-x(u)) \mid \mathcal{F}_{j-1}] \\
&\qquad \qquad + \mathbb{E}[p(x(u), t(u,j), k-1-h(u, x(u),j-1)) \cdot \mathbb{I}(x(j) \ge 1-x(u)) \mid \mathcal{F}_{j-1}].
\end{align*}
Thus, if $j$ is adjacent to $u$ in $G$, 
\begin{align}
    &\mathbb{E}[\II_u \mid \mathcal{F}_{j}] - \mathbb{E}[\II_u \mid \mathcal{F}_{j-1}] \nonumber\\
    &= \mathbb{E}[(p(x(u), t(u,j), k-1-h(u, x(u),j-1))-p(x(u), t(u,j), k-h(u, x(u),j-1)))\nonumber\\
    &\qquad \qquad \qquad \qquad \qquad \qquad \qquad \qquad \cdot (\mathbb{I}(x(j) \ge 1-x(u))-x(u)) \mid \mathcal{F}_{j-1}]. \label{eq:Iu}
\end{align}

Let $\delta_k(j,u) = p(x(u), t(u,j), k-1-h(u, x(u),j-1))-p(x(u), t(u,j), k-h(u, x(u),j-1))$. Note that $h(u,x(u),j-1)$ is $\sigma(\mathcal{F}_{j-1},x(u))$-measurable. In particular, if $u\le j-1$, then $\delta_k(j,u)$ is $\mathcal{F}_{j-1}$-measurable, and otherwise it is determined by $\mathcal{F}_{j-1}$ and the (independent) choice of $x(u)$, and we have 
\begin{align}
\mathbb{E}[Y_j^2 \mid \mathcal{F}_{j-1}] &= \mathbb{E}\left[\left(\sum_{u\in N(j)} \delta_{k}(j,u) (\mathbb{I}(x(j) \ge 1-x(u))-x(u))\right)^2 \bigg | \mathcal{F}_{j-1}\right] \nonumber \\
&\le 2\mathbb{E}_{x(j)}\left[\left(\sum_{u\in N(j), u\le j-1} \delta_{k}(j,u) (\mathbb{I}(x(j) \ge 1-x(u))-x(u))\right)^2\right] \nonumber \\
&\qquad \qquad + 2\mathbb{E}_{x(j),x(u)}\left[\left(\sum_{u\in N(j), u> j-1} \delta_{k}(j,u) (\mathbb{I}(x(j) \ge 1-x(u))-x(u))\right)^2\right], \label{eq:Yjdecomp}
\end{align}
where in the first summand we are averaging over the uniform random variable $x(j)$, and in the second summand we are averaging over independent $x(j)$ and $x(u)$ for all $u>j-1,u\in N(j)$. 

Let 
\begin{equation}
    A_1(j):=\mathbb{E}_{x(j)}\left[\left(\sum_{u\in N(j), u\le j-1} \delta_{k}(j,u) (\mathbb{I}(x(j) \ge 1-x(u))-x(u))\right)^2\right], \label{eq:A_1}
\end{equation}
and let 
\begin{equation}
    A_2(j) := \mathbb{E}_{x(j),x(u)}\left[\left(\sum_{u\in N(j), u> j-1} \delta_{k}(j,u) (\mathbb{I}(x(j) \ge 1-x(u))-x(u))\right)^2\right].\label{eq:A_2}
\end{equation}
 
In the next subsection, we will derive some useful bounds on $\delta_k(j,u)$ before turning to the upper bounds for $\sum_j A_1(j)$ and $\sum_j A_2(j)$, which complete the proof of Lemma \ref{lem:var-proxy}.%

\subsubsection{Bounds on $\delta_k(j,u)$}

We will now derive some useful bounds for $\delta_{k}(j,u)$ for fixed $u$. For convenience in notation, let $h = k-1-h(u,x(u),j-1)$, $x=x(u)$ and $t=t(u,j-1)$. Note that $m! = \Theta(\sqrt{m} (m/e)^m)$ for all $m\ge 1$ from Stirling's approximation. Letting $\alpha = h/(t-1)$, when $0<h<t-1$, we have %
\begin{align} \label{eq:stirling}
    |\delta_{k}(j,u)| &= |p(x,t,h)-p(x,t,h+1)| \nonumber \\ 
    &= \frac{t!}{(h+1)!(t-h)!} x^h (1-x)^{t-h-1} \left | (1-x)(h+1)-x(t-h)\right | \nonumber \\ 
    &\ll \frac{|t(\alpha-x)+1-\alpha-x|}{(\alpha(1-\alpha)t)^{3/2}} \exp\big([\alpha \log(x/\alpha)+(1-\alpha)\log((1-x)/(1-\alpha))](t-1)\big) \nonumber \\
    &\ll \frac{|t(\alpha-x)|+3}{(\alpha(1-\alpha)t)^{3/2}} \exp\big([\alpha \log(x/\alpha)+(1-\alpha)\log((1-x)/(1-\alpha))](t-1)\big).
\end{align}

We also have the trivial bound $|\delta_k(j,u)| \le 1$, since $\delta_k(j,u)$ is the difference of two probabilities. 

\begin{claim}
\label{claim:fbound}
There exists an absolute constant $c>0$ such that the following holds. For each $x\in [0,1]$, the function $f(\alpha) = \alpha \log(x/\alpha)+(1-\alpha)\log((1-x)/(1-\alpha))$ satisfies 
\begin{equation}
    f(\alpha) \le -c \min(|x-\alpha|, (1/x+1/(1-x))(x-\alpha)^2).\label{eq:fbound}
\end{equation}
\end{claim}
\begin{proof}
Note that 
\[
f'(\alpha) = \log \frac{x}{\alpha} - \log \frac{1-x}{1-\alpha}, \qquad f''(\alpha) = - \frac{1}{\alpha} - \frac{1}{1-\alpha}.
\]

Thus, by Taylor's theorem, noting that $f(x)=f'(x)=0$,
\begin{equation}
f(\alpha) = f(x)+f'(x)(\alpha-x)+ f''(\alpha')(\alpha-x)^2 = -\left(\frac{1}{\alpha'}+\frac{1}{1-\alpha'}\right) (x-\alpha)^2,\label{eq:secondorder}
\end{equation}
for some $\alpha'$ between $\alpha$ and $x$. %

By symmetry about $1/2$, we can assume without loss of generality that $x\le 1/2$. %
If $\alpha \le 2x$, then $\frac{1}{\alpha'}+\frac{1}{1-\alpha'} \ge c\left(\frac{1}{x}+\frac{1}{1-x}\right)$ for any $\alpha'$ between $\alpha$ and $x$, so (\ref{eq:secondorder}) implies (\ref{eq:fbound}). Otherwise, $\alpha>2x$, in which case we can verify that for some constants $c_1,c_2>0$,
\[
f(\alpha) = f(2x)+f(\alpha)-f(2x) \le f(2x) + (\alpha-2x)f'(2x) \le - c_1x - c_2 (\alpha-2x) \le -\min(c_1,c_2) |x-\alpha|. \qedhere
\]
\end{proof}

\subsubsection{Bounding $\sum_j A_1(j)$ and $\sum_j A_2(j)$}

Recall that $$A_1(j)=\mathbb{E}_{x(j)}\left[\left(\sum_{u\in N(j), u\le j-1} \delta_{k}(j,u) (\mathbb{I}(x(j) \ge 1-x(u))-x(u))\right)^2\right].$$

\begin{claim}\label{claim:high-prob}
\begin{enumerate}
    \item[(a)] With probability at least $1-n^2\exp(-c\kappa)$ over the random permutation $\pi$, we have that 
    \begin{equation}
        \left||N(v)\cap [j]| - d\frac{j}{n}\right|\le  \max\left(\sqrt{\kappa \frac{d}{n}\min(n-j,j)},\kappa\right) \label{eq:Nv}
    \end{equation} for all $v\in [n]$ and $j\in [n]$.
    \item[(b)] With probability at least $1-n^4\exp(-c\kappa)$ (over the randomness of all $x(u)$ with $u\in [n]$), we have that for any interval $I\subseteq [0,1]$ of length at least $2/d$, 
    \begin{equation}
        \big|\left|\{z\in N(v)\cap [j]:x(z)\in I\}\right| - d|I|j/n\big| \le e_{|I|} \label{eq:xz}
    \end{equation} %
    for all $v\in [n]$ and $j\in [n]$, where $$e_{r} = \kappa + \sqrt{\kappa drj/n}.$$%
\end{enumerate}
\end{claim}

\begin{proof}
We first prove the bound in (a). Note that under the random permutation $\pi$, $|N(v)\cap [j]|$ follows the hypergeometric distribution with parameters $(n,d,j)$. By Hoeffding's Theorem (Theorem 4 of \cite{Hoeffding}, that the hypergeometric distribution concentrates at least as much as the corresponding binomial distribution), we have that $$\mathbb{P}\left(\left||N(v)\cap [j]| - d\frac{j}{n}\right| > \theta d\frac{j}{n}\right) \le \exp(-\min(\theta/3,\theta^2/2)dj/n).$$
In particular, if $\kappa < \sqrt{dj/n}$, then $$\mathbb{P}\left(\left||N(v)\cap [j]| - d\frac{j}{n}\right| >  \sqrt{\kappa dj/n}\right) \le \exp(-c\kappa).$$
If $\kappa\ge \sqrt{dj/n}\ge 1$, then 
$$\mathbb{P}\left(\left||N(v)\cap [j]| - d\frac{j}{n}\right| > \kappa \sqrt{dj/n}\right) \le \exp(-c\kappa \sqrt{dj/n})\le \exp(-c\kappa).$$
Finally, if $\kappa\geq \sqrt{dj/n}$ and $\sqrt{dj/n}<1$, then 
$$\mathbb{P}\left(\left||N(v)\cap [j]| - d\frac{j}{n}\right| > \kappa \right) \le \exp(-c\kappa).$$
Hence, in all cases, 
$$\mathbb{P}\left(\left||N(v)\cap [j]| - d\frac{j}{n}\right| >  \max(\sqrt{\kappa dj/n},\kappa) \right) \le \exp(-c\kappa).$$
Finally, one has $|N(v)\cap [j]|=d-|N(v)\cap [j+1,n]|$, and an identical argument gives that 
$$\mathbb{P}\left(\left||N(v)\cap [j]| - d\frac{j}{n}\right| >  \max(\sqrt{\kappa d(n-j)/n},\kappa) \right) \le \exp(-c\kappa).$$
The final bound follows from the union bound over all $v$ and $j$.

Next, we prove (b). Let $|N(v)\cap [j]|=N_{v,j}$. Conditional on $\pi$, the random variable $N_{v,j,\ell,\ell'}:=|\{z\in N(v)\cap [j]:x(z)\in [\ell/d,\ell'/d]\}|$ is a binomial random variable with parameters $(N_{v,j},(\ell'-\ell)/d)$, so 
\[
\mathbb{E}\left[\exp(\lambda N_{v,j,\ell,\ell'}) | \pi \right] = \left(\frac{\ell'-\ell}{d}(e^{\lambda}-1)+1\right)^{|N_{v,j}|}.
\]
Next, we note that the function $y\mapsto \left(\frac{\ell'-\ell}{d}(e^{\lambda}-1)+1\right)^y$ is convex, so by Hoeffding's Theorem (Theorem 4 of \cite{Hoeffding}), we have
\begin{align*}
\mathbb{E}\left[\left(\frac{\ell'-\ell}{d}(e^{\lambda}-1)+1\right)^{|N_{v,j}|}\right] &\le \left(\frac{j}{n}\left(\frac{\ell'-\ell}{d}(e^{\lambda}-1)+1\right)+1-\frac{j}{n}\right)^d\\
&= \mathbb{E}\left[\exp(\lambda Z)\right],
\end{align*}
where $Z$ is binomially distributed with parameters $(d,\frac{j(\ell'-\ell)}{dn})$. This yields Chernoff-type bounds for concentration of $N_{v,j,\ell,\ell'}$, from which similar to the proof of (a), we obtain that with probability at least $1-\exp(-c\kappa)$, 
\begin{equation}\label{eq:Nvr1}
\left||\{z\in N(v)\cap [j]:x(z)\in [\ell/d,\ell'/d]\}|-j(\ell'-\ell)/n\right| \le  \frac{1}{2} \max\left(\kappa,\sqrt{\kappa j(\ell'-\ell)/n}\right).
\end{equation}
By the union bound, we can guarantee this property for all $v,j,\ell,\ell'$ with probability at least $1-n^4\exp(-c\kappa)$. 

The final conclusion follows upon noticing that, given any interval $I$ of length at least $2/d$, $I$ is contained in an interval $[\ell/d,\ell'/d]$ of length $(\ell'-\ell)/d\le |I|+1/d$, and $I$ contains an interval $[\tilde{\ell}/d,\tilde{\ell}'/d]$ of length $(\tilde{\ell}'-\tilde{\ell})/d \ge |I|-1/d$.

\end{proof}

Recall that $\kappa = C \log n$ and $k_+ = C\max(k,\kappa)$ for a sufficiently large constant $C$. In the rest of the section, we will work under the event $\mathcal{E}$ that the events in Claim \ref{claim:high-prob} hold. Let $k'=\max(k,\kappa)$.
 
\begin{lemma} \label{lem:x-offrange}
Let $n$ be sufficiently large and $j\le n-C\kappa n/d$. We have $$\sum \delta_k(j,u) < n^{-10},$$
where the summation ranges over $u\in N(j)$ with $u\le j-1$ such that $|x(u)-k/d|> \sqrt{\kappa k_+}/d$ %
or $|h/(t-1)-x(u)| > \max\big(\sqrt{\kappa k_+}/(d\sqrt{1-j/n}),\kappa/(d(1-j/n))\big)$. 
\end{lemma}
\begin{proof}
Under the event $\mathcal{E}$, Claim~\ref{claim:high-prob}(b) tells us that %
$$\left| h(u,x(u),j-1)-x(u)(j-1)d/n\right|\le e_{x(u)},$$ so \begin{equation}\label{eq:h}\left|h - (k-1-x(j-1)d/n)\right|\le e_{x}.\end{equation}
By Claim~\ref{claim:high-prob}(a),
\begin{equation}\label{eq:t}\left|t - (d - d(j-1)/n)\right| \le  \max(\kappa, \sqrt{\kappa d\min(n-(j-1),j-1)/n}).\end{equation}

From (\ref{eq:h}) and (\ref{eq:t}), we have \begin{equation} \label{eq:rangeh-x} h/(t-1)-x \in \left[\frac{M_1-E_1}{M_2+E_2},\frac{M_1+E_1}{M_2-E_2}\right],\end{equation} where
\begin{align*}
M_1 &= k-1-x(d-1), \,\, & M_2 &= d-1-d(j-1)/n, \\
E_1 &=  e_x +  x (\kappa+\sqrt{\kappa d\min(n-(j-1),j-1)/n}), \,\, & E_2 &= \kappa+\sqrt{\kappa d\min(n-(j-1),j-1)/n}.
\end{align*}

Note that 
\[
E_2 \le \kappa + \sqrt{\kappa d(n-j+1)/n} \le \kappa +  \sqrt{\kappa(M_2+1)} \le M_2/2,
\]
where we use the fact that $M_2 \ge C\kappa-1$ for $j\le n-C\kappa n/d$. This, along with \eqref{eq:t}, implies that
\begin{equation}
    t \geq M_2+1-E_2 = 1+\frac{1}{2}M_2> \frac{1}{2} d(1-(j-1)/n). \label{eq:eqt}
\end{equation}
Furthermore, recall that
\[
e_x = \kappa+\sqrt{\kappa dx(j-1)/n}. %
\]

Let $c$ be the constant in Claim \ref{claim:fbound} and let $C'=cC/6400$. Assume that $C$ is sufficiently large so $C'>1$. By Claim~\ref{claim:fbound}, if $|h/(t-1)-x|>\frac{1}{64C'}\kappa/(d(1-j/n))$ and $|h/(t-1)-x|^2/x >\frac{1}{64C'} \kappa/(d(1-j/n))$, %
then 
\begin{align}
&\exp\big([\alpha \log(x/\alpha)+(1-\alpha)\log((1-x)/(1-\alpha))](t-1)\big) \nonumber \\
&\le \exp\left(-c d(1-j/n) \min\left(|h/(t-1)-x|, \frac{1}{x}|h/(t-1)-x|^2\right)\right) \nonumber \\
&< n^{-15} \label{eq:poly-small}, 
\end{align}
upon choosing the constant $C$ in the definition of $\kappa$ sufficiently large. 

We will next show that the conditions of (\ref{eq:poly-small}) (and hence the conclusion) hold if $|x-k/d|>\sqrt{\kappa k_+}/d$ or $|h/(t-1)-x|>\max\big(\sqrt{\kappa k_+}/(d\sqrt{1-j/n}),\kappa/(d(1-j/n))\big)$. Given this, by (\ref{eq:stirling}), we have that for $u\in N(j)$ with $u\le j-1$, and either $|x(u)-k/d|\ge \sqrt{\kappa k_+}/{d} \ge \sqrt{C'\kappa k'}/{d}$ or $|h/(t-1)-x(u)| > \max\big(\sqrt{\kappa k_+}/(d\sqrt{1-j/n}),\kappa/(d(1-j/n))\big)$, %
$$\delta_k(j,u) \ll \frac{|t(\alpha-x)|+3}{(\alpha(1-\alpha)t)^{3/2}} n^{-15} < n^{-12}.$$
Summing over at most $n$ such $u$'s, we obtain the desired bound. 

We now turn to verify the conditions of (\ref{eq:poly-small}) when $|x-k/d|>\sqrt{\kappa k_+}/d$ or $|h/(t-1)-x|>\max\big(\sqrt{\kappa k_+}/(d\sqrt{1-j/n}),\kappa/(d(1-j/n))\big)$. Let $\eta := |x-k/d|$, and let $\eta_0 := \sqrt{C'\kappa k'}/d$.

{\noindent \textbf{Case 1:}} $|x-k/d|=\eta \ge \eta_0$.

From \eqref{eq:rangeh-x}, we first claim that if $\eta \ge \eta_0$ then $|E_1|<|M_1|/2$. Indeed, if $x < 64k'/d$, then %
\begin{align*}
    |M_1| &> \eta d-2 \ge \sqrt{C' \kappa k'}-2, \,\, & %
    |E_1| &< 2(\kappa+\sqrt{\kappa d x}) \le 2(\kappa + 8\sqrt{\kappa k'}) \le 18\sqrt{\kappa k'}. %
\end{align*} 
Otherwise, if $x\ge 64k'/d$, then $\eta < x$, and 
\begin{align*}
    |M_1| &> \frac{1}{2} d x, \,\, & %
    |E_1| &< 2(\kappa+\sqrt{\kappa dx}). %
\end{align*} 
In both cases, recalling that $k'=\max(k,\kappa)\ge \kappa$, we indeed have for a sufficiently large $C$ that $|E_1| < |M_1|/2$. Combining this estimate with the fact that $E_2\leq M_2/2$ yields
\begin{equation}
    4\eta/(1-j/n) > |h/(t-1)-x| > \frac{1}{4} \eta / (1-j/n). \label{eq:eqh}
\end{equation}
In particular, since $k'\ge \kappa$, we have 
\begin{equation}\label{eq:h2}
|h/(t-1)-x| > \frac{1}{4} \eta_0 / (1-j/n) = \frac{1}{4}\sqrt{C'\kappa k'}/(d(1-j/n)) \ge \frac{1}{64C'}\kappa /(d(1-j/n)). 
\end{equation}

For $\alpha = h/(t-1)$, if $|x-k/d|=\eta \ge \sqrt{C'\kappa k'}/d$, then 
\[
\frac{1}{x} |h/(t-1)-x|^2 \ge \frac{1}{16}\frac{1}{k/d+\eta} \frac{\eta^2}{(1-j/n)^2}.
\]
Thus, if $\eta > (k+\kappa)/d$, then %
\[
\frac{1}{x} |h/(t-1)-x|^2 \ge \frac{\eta}{32(1-j/n)^2} > \frac{k+\kappa}{32d(1-j/n)^2} > \frac{1}{64C'}\frac{\kappa}{d(1-j/n)},
\]
and otherwise if $\eta \le (k+\kappa)/d$, then 
\[
\frac{1}{x} |h/(t-1)-x|^2 \ge \frac{1}{16}\frac{d}{2k+\kappa} \frac{\eta_0^2}{(1-j/n)^2}\ge \frac{1}{16}\frac{C'k'}{2k+\kappa}\frac{\kappa}{d(1-j/n)^2}\ge \frac{1}{48}\frac{\kappa}{d(1-j/n)}>\frac{1}{64C'}\frac{\kappa}{d(1-j/n)}.
\]

The estimates above and (\ref{eq:h2}) yield that the conditions of (\ref{eq:poly-small}) hold when $|x-k/d|\ge \sqrt{C'\kappa k'}/d$. %

{\noindent \textbf{Case 2:}} $|x-k/d|\le \eta_0 = \sqrt{C'\kappa k'}/d$ and $|h/(t-1)-x| > \max\big(\sqrt{\kappa k_+}/(d\sqrt{1-j/n}),\kappa/(d(1-j/n))\big)$.

 Since $|x-k/d|\le \sqrt{C'\kappa k'}/d$, we have $x \le 2\max(k/d, \sqrt{C'\kappa k'}/d)$. Then 
\begin{align*}
\frac{1}{x}|h/(t-1)-x|^2 &\ge \frac{1}{2\max(k/d, \sqrt{C'\kappa k'}/d)} \cdot \frac{\kappa k_+}{d^2(1-j/n)} \ge \frac{1}{2C'(k+\kappa)/d} \cdot \frac{\kappa (\kappa+k)/2}{d^2(1-j/n)}\\
&>\frac{1}{64C'}\frac{\kappa}{d(1-j/n)}.
\end{align*}
Furthermore, 
\[
|h/(t-1)-x| > \frac{\kappa}{d(1-j/n)} > \frac{1}{64C'}\frac{\kappa}{d(1-j/n)}.
\]This yields the conditions of (\ref{eq:poly-small}) in this case and completes the proof of the lemma. 

\end{proof} 

We next return to the task of bounding $\sum_j A_1(j)$. %
Let $\mathcal{U}_j$ be the set of $u$ with $|x(u)-k/d| \le \sqrt{\kappa k_+}/d$ and $|h/(t-1)-x(u)| \le \max\big(\sqrt{\kappa k_+}/(d\sqrt{1-j/n}),\kappa/(d(1-j/n))\big)$. Note that by symmetry in degrees of $H$ about $d/2$, we can and will assume that $k\le d/2$. Let $C''$ be a sufficiently large constant whose choice depends on $C$. In the following, we first consider values of $j$ such that $j\le n-C''n\kappa/d$. For these $j$ we have $d\sqrt{1-j/n} \ge \sqrt{C''\kappa d} \ge \sqrt{C\kappa k_+}$, since $d=\omega(\log n)$ and $d(1-j/n) \ge C''\kappa$. %
Thus, for $C$ sufficiently large and $k\le d/2$, we have that $\min(1-\alpha,1-k/d)\ge 1/4$ (where we recall that $\alpha = h/(t-1)$).

By Claim~\ref{claim:high-prob}(b), the number of $u\in N(j)\cap [j-1]$ with $|x(u)-k/d|\le  \sqrt{\kappa k_+}/d$ is at most $\frac{dj}{n} \cdot 2\frac{\sqrt{\kappa k_+}}{d} + \left(\kappa+\sqrt{2\kappa \sqrt{\kappa k_+}j/n}\right) \le 4\left(\kappa+\sqrt{\kappa k_+ j/n}\right)$%
. Furthermore, recall that 
\begin{align*}
A_1(j) &= \mathbb{E}_{x(j)}\left[\left(\sum_{u\in N(j), u\le j-1} \delta_{k}(j,u) (\mathbb{I}(x(j) \ge 1-x(u))-x(u))\right)^2\right]\\
&\ll \mathbb{E}_{x(j)}\left[\left(\sum_{u\in N(j), u\le j-1, u\notin \mathcal{U}_j} \delta_{k}(j,u) \right)^2\right] \\
&\qquad \qquad + \mathbb{E}_{x(j)}\left[\left(\sum_{u\in N(j), u\le j-1, u \in \mathcal{U}_j} \delta_{k}(j,u) (\mathbb{I}(x(j) \ge 1-x(u))-x(u)) \right)^2\right].
\end{align*}
If $x(j) < 1 - k/d - \sqrt{\kappa k_+}/d$, then the second term can be bounded by 
\[
\mathbb{E}_{x(j)}\left[\left(\sum_{u\in N(j), u\le j-1, u\in \mathcal{U}_j} \delta_{k}(j,u) (k+\sqrt{\kappa k_+})/d \right)^2\right].
\]
Since $x(j) \ge 1-k/d-\sqrt{\kappa k_+}/d$ occurs with probability at most $k/d+\sqrt{\kappa k_+}/d$, we conclude that
\begin{align}
A_1(j) &\ll \mathbb{E}_{x(j)}\left[\left(\sum_{u\in N(j), u\le j-1, u\notin \mathcal{U}_j} \delta_{k}(j,u) \right)^2\right]  \nonumber\\
&\qquad \qquad + \mathbb{E}_{x(j)}\left[\left(\sum_{u\in N(j), u\le j-1, u\in \mathcal{U}_j} \delta_{k}(j,u)\right)^2\right] \left(\frac{(k+\sqrt{\kappa k_+})^2}{d^2}+\frac{k+\sqrt{\kappa k_+}}{d}\right) \nonumber\\
&\ll n^{-20} + \frac{k+\sqrt{\kappa k_+}}{d}\mathbb{E}_{x(j)}\left[\left(\sum_{u\in N(j), u\le j-1, u\in \mathcal{U}_j} \delta_{k}(j,u)\right)^2\right],\label{eq:A1bound}
\end{align}
where the the bound on the first term follows from Lemma \ref{lem:x-offrange}. In particular, in the following part of the argument, we only need to restrict our attention to the case where $|x-k/d|\le \sqrt{\kappa k_+}/d$ and $
|h/(t-1)-x| \le \max\big(\sqrt{\kappa k_+}/(d\sqrt{1-j/n}),\kappa/(d(1-j/n))\big)$, %
so $\alpha \in [k/d\pm  2\max\big(\sqrt{\kappa k_+}/(d\sqrt{1-j/n}),\kappa/(d(1-j/n))\big)]$.

Condition on $x(v)$ %
with $|x(v)-k/d| > \sqrt{\kappa k_+}/d$. Since the event in Claim~\ref{claim:high-prob}(b) holds, the $x(v)$ with $|x(v)-k/d|>\sqrt{\kappa k_+}/d$ partition $[0,1] \setminus [k/d \pm \sqrt{\kappa k_+}/d]$ into consecutive intervals each of length at most $\kappa/d$, and the union of any $q$ consecutive intervals has length at most $(2q+4\kappa)/d$%
. Indeed, this follows since any interval of length $(2q+4\kappa)/d$ contains at least $(2q+4\kappa)- (\kappa + \sqrt{\kappa (2q+4\kappa)})>q$ elements $x(v)$. 

Since $|x(u)-k/d|\le \sqrt {\kappa k_+/d}$, if $|h/(t-1)-x(u)|\le  \max\big(\sqrt{\kappa k_+}/(d\sqrt{1-j/n}),\kappa/(d(1-j/n))\big)$, then $h/(t-1)$ is contained in an interval centered at $k/d$ of length $$w:= 2\sqrt{\kappa k_+/d} + 2\max\big(\sqrt{\kappa k_+}/(d\sqrt{1-j/n}),\kappa/(d(1-j/n))\big).$$ Thus, $h$ must be contained in a fixed interval of length at most $(t-1)w$, which implies that $u$ is contained in a (fixed) union of at most $(t-1)w$ consecutive intervals in the partition by $x(v)$ with $|x(v)-k/d|>\sqrt{\kappa k_+}/d$. (Note that each possible value of $h$ determines an interval in the partition that $x(u)$ must lie in.) Hence, $x(u)$ must be contained in a fixed interval of length at most %
\[
\Big(4\kappa + 2(t-1)w\Big)/d < 12\kappa/d + 16 \sqrt{\kappa k_+ (1-j/n)}/d,
\]
where we have used the fact that $t<2d(n-j)/n$. The number of such $u$ in $N(j)\cap [j-1]$ is at most $32(\kappa + \sqrt{\kappa k_+(1-j/n)})$. 

Letting $\bar{t} = d(n-j)/n$, we have $t=\Theta(\bar t)$ by Claim~\ref{claim:high-prob}(a), as reasoned in the proof of Lemma~\ref{lem:x-offrange}. By (\ref{eq:stirling}) and recalling that $\min(1-\alpha,1-k/d)\ge 1/4$, we have the bound
\begin{equation}
    \label{eq:deltak-1}
\delta_k(j,u) \ll \frac{3+t \max\big(\sqrt{\kappa k_+}/(d\sqrt{1-j/n}),\kappa/(d(1-j/n))\big)}{(\max(1,\alpha t))^{3/2}}. %
\end{equation}
Note that if $k<4\sqrt{\kappa k_+}/\sqrt{1-j/n}$, then $\bar{t} k/d = k(1-j/n) \ll \kappa^2$. Similarly, if $k< 4\kappa/(1-j/n)$, then $\bar{t}k/d = k(1-j/n) \ll \kappa$. On the other hand, if $k \ge \max\big(4\sqrt{\kappa k_+}/\sqrt{1-j/n},4\kappa/(1-j/n)\big)$, then 
\[
\alpha \ge k/d-\max\big(\sqrt{\kappa k_+}/(d\sqrt{1-j/n}),\kappa/(d(1-j/n))\big) \ge k/d - k/(2d) = k/(2d).
\]
Thus, in both cases, $\max(1,\alpha t)\gg \bar{t} \kappa^{-2} k/d$. Hence,
\begin{align}
    \label{eq:deltak}
\delta_k(j,u) &\ll \kappa^3 \frac{3+\bar{t} \max\big(\sqrt{\kappa k_+}/(d\sqrt{1-j/n}),\kappa/(d(1-j/n))\big)}{(k\bar{t}/d)^{3/2}} \nonumber \\
&\ll \kappa^3 \frac{\kappa+\sqrt{\kappa k_+(1-j/n)}}{(k(1-j/n))^{3/2}}. %
\end{align}
We also have the trivial bound 
\[
\delta_k(j,u)\le 1. 
\]

Hence, we have 
\begin{align}
    &\kappa^{-6}\sum_{j=1}^{n-C''\kappa n/d} A_1(j) - n^{-19}\nonumber\\
    &\ll \sum_{j=1}^{n-C''\kappa n/d}\frac{k+\sqrt{\kappa k_+}}{d}\left((\kappa + \sqrt{\kappa k_+(1-j/n)}) \cdot \min\left(1,\frac{ \kappa+\sqrt{\kappa k_+(1-j/n)}}{(k(1-j/n))^{3/2}}\right)\right)^2 \nonumber \\
    &\ll \frac{k+\sqrt{\kappa k_+}}{d}\sum_{j\in [1,n-C''\kappa n/d]:k(1-j/n)>1}\left((\kappa+\sqrt{\kappa k(1-j/n)}) \cdot \frac{\kappa+\sqrt{\kappa k_+(1-j/n)}}{(k(1-j/n))^{3/2}}\right)^2 \nonumber\\
    &\qquad\qquad\qquad\qquad + \frac{k+\sqrt{\kappa k_+}}{d}\sum_{j\in [1,n-C''\kappa n/d]:k(1-j/n)\le 1}\kappa^2 \nonumber\\
    & \ll \frac{k+\sqrt{\kappa k_+}}{d} \bigg(\kappa^4 \sum_{j<n-n/k}\frac{n^3}{k^3(n-j)^3} + \kappa^3 \sum_{j<n-n/k}\frac{(k+\kappa)n^2}{k^3(n-j)^2}  \nonumber\\
    &\qquad \qquad \qquad \qquad \qquad \qquad + \kappa^3 \sum_{j<n-n/k} \frac{n^2}{k^2(n-j)^2} + \sum_{j<n-n/k}\kappa^2 \frac{(k+\kappa)n}{k^2(n-j)} + \kappa^2 \frac{n}{k}\bigg) \nonumber  \nonumber\\
    &\ll \kappa^4 \frac{k+\kappa}{d}\cdot \frac{n^3}{(n/k)^2k^3} + \kappa^3 \frac{(k+\kappa)^2n^2}{k^3(n/k)d} + \kappa^3 \frac{k+\kappa}{d}\cdot \frac{n^2}{k^2(n/k)}  \nonumber\\
    &\qquad \qquad \qquad \qquad \qquad \qquad + \frac{(\log k)\kappa^2(k+\kappa)^2n}{dk^2} + \kappa^2 \frac{n(k+\kappa)}{dk} \nonumber \\
    &\ll \kappa^5 n/d. \label{eq:mid}
\end{align} 
Here, we note that $k(1-j/n)>1$ if and only if $j<n-n/k$. The first inequality follows from  applying \eqref{eq:A1bound} and \eqref{eq:deltak}. The second inequality follows from breaking the summation into the ranges $j<n-n/k$ and $j\ge n-n/k$ and bounding each separately. Specifically, in the range $j<n-n/k$, we use the bound
\[
\kappa + \sqrt{\kappa k_+(1-j/n)} \ll \kappa + \sqrt{\kappa (\kappa + k)(1-j/n)} \ll \kappa + \sqrt{\kappa k(1-j/n)},
\]
and in the range $j \ge n-n/k$, we use the bound
\[
(\kappa + \sqrt{\kappa k_+(1-j/n)}) \cdot \min\left(1,\frac{3+ \sqrt{(1-j/n)\kappa k_+}}{(k(1-j/n))^{3/2}}\right) \ll \kappa + \sqrt{\kappa k_+(1-j/n)} \ll \kappa + \sqrt{\kappa k(1-j/n)} \ll \kappa.
\]
The third inequality follows from applications of the inequality $(a+b)^2 \ll a^2 + b^2$. The fourth inequality follows from the elementary bounds $\sum_{j<n-n/k}\frac{1}{(n-j)^{s}} \ll (n/k)^{-s+1}$ for $s\ge 2$, and $\sum_{j<n-n/k}\frac{1}{n-j} \ll \log k$, as well as the observation $k+\sqrt{\kappa k_+} \ll k+\kappa$. In the last inequality, we use the bound $(k+\kappa)/k \ll \kappa$. 

Finally, in the range $j\ge n-C''\kappa n/d$, by Lemma \ref{lem:sup-Yi}, we have that $A_1(j) \ll \kappa^2$ and hence 
\begin{equation}\label{eq:largest}
    \sum_{j=n-C''\kappa n/d}^{n} A_1(j) \ll \kappa^2 \cdot \kappa n/d = \kappa^3 n/d.
\end{equation}

Combining (\ref{eq:mid}) and (\ref{eq:largest}), we obtain 
\begin{equation}\label{eq:bound-A1}
    \sum_{j=1}^{n}A_1(j) \ll (\log n)^{11}n/d.
\end{equation}

The control of $$A_2(j) = \mathbb{E}_{x(j),x(u)}\left[\left(\sum_{u\in N(j), u> j-1} \delta_{k}(j,u) (\mathbb{I}(x(j) > 1-x(u))-x(u))\right)^2\right]$$ is very similar, and we omit the details of the argument. One first considers a realization of the random variables $x(u)$ for $u\in N(j),u>j-1$. The same argument used in bounding $A_1$ then allows us to prove the same bound on $A_2$:
\begin{equation}\label{eq:bound-A2}
    \sum_{j=1}^{n}A_2(j) \ll (\log n)^{11}n/d.
\end{equation}

We are now ready to finish the proof of Lemma \ref{lem:var-proxy} and hence the proof of Theorem \ref{thm:main}. 

\begin{proof}[Proof of Lemma \ref{lem:var-proxy}]
By \eqref{eq:Mn} and \eqref{eq:Yjdecomp}, we have 
\[
M_n \ll \sum_{j=1}^{n}A_1(j)+\sum_{j=1}^{n}A_2(j) \ll (\log n)^{11}n/d,
\]
where in the last inequality we have used \eqref{eq:bound-A1} and \eqref{eq:bound-A2}. 
\end{proof}

\end{document}